\documentclass{amsart}

\usepackage{amsfonts,amssymb}
\usepackage{amsmath}
\usepackage{algorithm}

\newtheorem{thm}{Theorem}

\newtheorem{cor}[thm]{Corollary}

\newtheorem{definition}{Definition}
\newtheorem{example}{Example}
\newtheorem{remark}{Remark}

\begin{document}

\title[Extremal nonnegative solutions of the $M$-tensor equation]{Computing the extremal nonnegative solutions of the $M$-tensor equation with a nonnegative right side vector}

\author{Chun-Hua Guo}
\address{Department of Mathematics and Statistics, 
University of Regina, Regina,
Saskatchewan S4S 0A2, Canada} 
\email{chun-hua.guo@uregina.ca}

\subjclass{Primary 65F10; Secondary 15A69, 65H10}

\keywords{$Z$-tensor; $M$-tensor; Maximal nonnegative solution; Minimal nonnegative solution; Iterative methods; Linear convergence.}

\begin{abstract}

We consider the tensor equation whose coefficient tensor is a nonsingular $M$-tensor and whose right side vector is nonnegative. 
Such a tensor equation may have a large number of nonnegative solutions. It is already known that the tensor equation has a maximal nonnegative solution and a minimal nonnegative solution (called extremal solutions collectively). 
However, the existing proofs do not show how the extremal solutions can be computed. The existing numerical methods can find one of the nonnegative solutions, without knowing whether the computed solution is an extremal solution. In this paper, we present new proofs for the existence of extremal solutions. Our proofs are much  shorter than existing ones and more importantly they give numerical methods that can compute the extremal solutions. Linear convergence of these numerical methods is also proved under mild assumptions. Some of our discussions also allow the coefficient tensor to be a $Z$-tensor or allow the right side vector to have some negative elements. 

\end{abstract}

\maketitle

\section{Introduction}
\label{sec1}
We consider the tensor equation 
\begin{equation}\label{Teq}
{\mathcal A}x^{m-1}=b,
\end{equation}
where $x, b\in {\mathbb R}^n$ and ${\mathcal A}$ is a real-valued $m$th-order $n$-dimensional tensor that has the form 
$$
{\mathcal A}=(a_{i_1i_2\ldots i_m}), \quad a_{i_1i_2\ldots i_m}\in {\mathbb R}, \quad 1\le i_1,i_2,\ldots, i_m\le n, 
$$
and ${\mathcal A}x^{m-1}\in {\mathbb R}^n$ has elements 
$$
({\mathcal A}x^{m-1})_i=\sum_{i_2,\ldots, i_m=1}^{n} a_{ii_2\ldots i_m} x_{i_2}\cdots x_{i_m}, i=1,2, \ldots, n. 
$$
Later on, we will also use $\{x_k\}$ to denote a sequence of vectors. 
The $i$th element of a vector $x$ will then be denoted by $(x)_i$ (instead of $x_i$) to avoid confusion. 
The element $a_{i_1i_2\ldots i_m}$ of the tensor ${\mathcal A}$ will also be denoted by ${\mathcal A}_{i_1i_2\ldots i_m}$.
The tensor equation \eqref{Teq} is also called a multilinear system of equations. It appears in many applications including data mining, numerical partial differential equations, and tensor complementarity problems; see \cite{LXG21} and the references cited therein. 

We denote the set of all  real-valued $m$th-order $n$-dimensional tensors by 
${\mathbb R}^{[m,n]}$. A tensor ${\mathcal A}=(a_{i_1i_2\ldots i_m})\in {\mathbb R}^{[m,n]}$ is called semi-symmetric \cite{NQ15}  if $a_{i j_2 \ldots j_m}=a_{i i_2 \ldots i_m}$, $1\le i\le n$, $j_2\ldots
j_m$ is any permutation of $i_2\ldots i_m$, $1\le i_2, \ldots, i_m\le n$.
For any tensor ${\mathcal A}=(a_{i_1i_2\ldots i_m})$, we have a semi-symmetric tensor  $\bar{{\mathcal A}}=(\bar{a}_{i_1i_2\ldots i_m})$ defined by 
$$
\bar{a}_{i_{1} i_{2}\ldots i_{m}}=\frac{1}{(m-1)!}%
\sum_{j_{2}\ldots j_{m}}a_{i_{1}j_{2}\ldots j_{m}},
$$
where $j_{2}\ldots j_{m}$ is any permutation of $i_{2}\ldots i_{m}.$ 
Then ${\mathcal A}x^{m-1}=\bar{{\mathcal A}}x^{m-1}$ for all $x\in {\mathbb R}^n$, and 
$({\mathcal A}x^{m-1})'=(m-1)\bar{{\mathcal A}}x^{m-2}$ (see \cite{LGL17,NQ15}), where for any ${\mathcal A}\in {\mathbb R}^{[m,n]}$, the $(i,j)$ element of ${\mathcal A}x^{m-2}\in {\mathbb R}^{n\times n}$ is defined by 
$$
({\mathcal A}x^{m-2})_{ij}=\sum_{i_3, \ldots, i_m=1}^n a_{iji_3\ldots i_m}{x}_{i_3}\cdots {x}_{i_m}. 
$$

We call $\lambda\in \mathbb R$ an eigenvalue and $x\in {\mathbb R}^n\setminus \{0\}$ a corresponding eigenvector of 
${\mathcal A}$ if \cite{QL17}
$$
{\mathcal A}x^{m-1}=\lambda x^{[m-1]}, 
$$
where for any positive real number $r$, $x^{[r]}\in {\mathbb R}^n$ is given by $(x^{[r]})_i=(x_i)^{r}$, $i=1,2,\ldots, n$. The spectral radius of ${\mathcal A}$  is the maximum modulus of its eigenvalues, and is denoted by  $\rho({\mathcal A})$. 

Let $[n]=\{1, 2, \ldots, n\}$. For $x, y\in {\mathbb R}^n$, we write $x\ge y$ if $x_i\ge y_i$ for all $i\in [n]$, 
and write $x> y$ if $x_i > y_i$ for all $i\in [n]$. If $x\ge 0$, we say $x$ is nonnegative; if $x>0$ we say $x$ is positive. 
The set ${\mathbb R}_+=\{x\in {\mathbb R}^n \ | \ x\ge 0\}$ will be used frequently. 
A solution $x$ of \eqref{Teq} is said to be a maximal nonnegative solution if $x\ge y$ for every solution $y\ge 0$; 
a solution $x$ of \eqref{Teq} is said to be a minimal nonnegative solution if $0\le x\le y$ for every solution $y\ge 0$. 
A tensor ${\mathcal A}$ is said to be nonnegative, denoted by ${\mathcal A}\ge 0$,  if all its elements are nonnegative. 
The identity tensor ${\mathcal I}\in {\mathbb R}^{[m,n]}$ is such that its diagonal elements are all ones and its off-diagonal elements are all zeros, i.e., 
${\mathcal I}_{ii\ldots i}=1$ for $i\in [n]$ and ${\mathcal I}_{i_1i_2\ldots i_m}=0$ elsewhere. 

\begin{definition}\cite{DQW13}
A tensor  ${\mathcal A}\in {\mathbb R}^{[m,n]}$   is called a $Z$-tensor if its off-diagonal elements are nonpositive. 
If ${\mathcal A}$ can be written as ${\mathcal A}=s{\mathcal I}-{\mathcal B}$ with ${\mathcal B}\ge 0$ and $s>\rho({\mathcal B})$, 
then the tensor ${\mathcal A}$ is called a nonsingular $M$-tensor.
\end{definition}

For ${\mathcal B}\ge 0$, $\rho({\mathcal B})$ can be found by the power method in \cite{NQZ09} and the Newton--Noda iteration 
in \cite{LGL17}. Thus, one can determine whether a $Z$-tensor is a nonsingular $M$-tensor using the definition. 
It is known \cite{DQW13} that a $Z$-tensor ${\mathcal A}$ is a nonsingular ${\mathcal M}$-tensor if and only if 
${\mathcal A}x^{m-1}>0$ for some $x>0$. It follows that the diagonal elements of a nonsingular ${\mathcal M}$-tensor are all positive.

Suppose ${\mathcal A}$ is a nonsingular $M$-tensor. When $b$ is positive, equation \eqref{Teq} has a unique positive solution \cite{DW16}. When $b$ is nonnegative, equation \eqref{Teq} has nonnegative solutions and moreover it has a minimal nonnegative solution \cite{LQX17} and  has a maximal nonnegative solution \cite{LGW20}. 

The following simple example is an extension of Example 1.1 in \cite{BHLZ21}. It shows that the number of nonnegative solutions could be huge when $b\ge 0$. 

\begin{example}\label{ex1}
Let $k\ge 1$ be an integer and let ${\mathcal A}=(a_{i_1i_2i_3i_4})\in {\mathbb R}^{[4, 2k]}$, where 
$a_{iiii}=1$ for $i\in [2k]$, $a_{2i-1,2i-1,2i-1,2i}=-2$ for $i\in [k]$, and all other elements are zero. 
It is clear that ${\mathcal A}$ is a $Z$-tensor with ${\mathcal A}x^3>0$ for $x=[3,1, \ldots, 3,1]^T$. 
So ${\mathcal A}$ is a nonsingular $M$-tensor. For $b=[0,1, \ldots,0,1]^T$, we find that the equation 
${\mathcal A}x^3=b$ has $2^k$ solutions $x=[x_1,x_2,\ldots, x_{2k-1}, x_{2k}]^T$, 
where for $i\in [k]$,  $[x_{2i-1}, x_{2i}]=[0,1]^T$ or $[2,1]^T$. 
The minimal nonnegative solution is $[0,1, \ldots, 0,1]^T$ and the maximal nonnegative solution is  $[2,1,\ldots, 2,1]^T$. 
\end{example}

Since equation \eqref{Teq} may have a large number of nonnegative solutions, it is unlikely that each of these solutions is of practical interest. Intuitively, the extremal nonnegative solutions are of particular interest.  By computing the maximal nonnegative solution, we can answer the question whether the equation has a positive solution. 
By computing the extremal nonnegative solutions, we get bounds for all other nonnegative solutions. 

In \cite{LQX17}, the tensor complementarity problem (TCP):
$$
x\ge 0, \quad {\mathcal A}x^{m-1}-b\ge 0, \quad x^T({\mathcal A}x^{m-1}-b)=0
$$
is considered. When ${\mathcal A}$ is a $Z$-tensor and $b\in {\mathbb R}_+$, it is shown in \cite{LQX17} that the solution set of the TCP is the same as the set of all nonnegative solutions of equation \eqref{Teq}. 
So a sparsest solution to the TCP is the minimal nonnegative solution of equation \eqref{Teq}.
The problem of finding a sparsest solution to the TCP is of practical interest \cite{LQX17}. 
Note that, for this application, the case $b\ge 0$ with some zero elements is much more interesting than the case $b>0$ (where the equation has a unique positive solution and there are no other nonnegative solutions). 

When ${\mathcal A}$ is a nonsingular $M$-tensor and $b>0$, the unique positive solution $x({\mathcal A}, b)$ of ${\mathcal A}x^{m-1}=b$
is the solution of particular interest. But some elements of $b$ may be very tiny. What happens to $x({\mathcal A}, b)$ if $b$ decreases monotonically towards a vector $b^{(0)}$ with 
some zero elements? We will see in the next section that $x({\mathcal A}, b)$ converges to the maximal nonnegative solution of the equation 
${\mathcal A}x^{m-1}=b^{(0)}$. 

In \cite{LGW20}, the authors state in the introduction that their purpose is to find the largest (maximal)  nonnegative solution. But their algorithms there only find a nonnegative solution, which is usually not maximal. Numerical methods have also been presented in \cite{BHLZ21,LGX20,LXG21}. Those methods can find a nonnegative solution, but the solution is usually not maximal and cannot be guaranteed to be minimal. 

While our main interest is on equation \eqref{Teq} with ${\mathcal A}$ being a nonsingular $M$-tensor and $b\ge 0$, some of our discussions also allow 
${\mathcal A}$ to be a $Z$-tensor or allow $b$ to have some negative elements. In Section \ref{sec2}, 
we present new proofs for the existence of extremal nonnegative solutions by using simple iterative methods. 
These iterative methods can actually compute the extremal solutions. Some other results also follow directly from the existence theorems and their proofs. 
In Section \ref{sec3}, we show how the simple iterations used in Section \ref{sec2} can be generalized to have faster convergence. 
Linear convergence of these iterative methods is also proved under mild assumptions. Some concluding remarks are given in Section \ref{sec4}.

\section{Existence of extremal nonnegative solutions} 
\label{sec2}

The following theorem is a main result in \cite{LQX17} (stated differently in \cite{LQX17}; see \cite[Theorem 3]{LQX17} and its proof). 
The proof there is based on several other results and does not show how the minimal nonnegative solution can be computed. 
Our new proof here is very short and it provides a way to compute the minimal solution. 
The approach we take here is similar to the approach we used many years ago for determining the existence of the minimal nonnegative solution of $M$-matrix algebraic Riccati equations \cite{G01, GL00}.

\begin{thm}\label{thmmin}
Let ${\mathcal A}$ be a $Z$-tensor and $b\in {\mathbb R}_+$. Suppose that  
$${\mathcal S}_g=\{x\in {\mathbb R}_+ \ | \   {\mathcal A}x^{m-1}\ge b\}\ne \emptyset. 
$$  
Then ${\mathcal S}_g$ has a minimal element that is also the minimal nonnegative solution of ${\mathcal A}x^{m-1}=b$. 
\end{thm}

\begin{proof}
We write ${\mathcal A}={\mathcal D}-{\mathcal B}$, where ${\mathcal D}$ is a diagonal tensor with positive diagonal elements and ${\mathcal B}\ge 0$. Equation \eqref{Teq}
 becomes 
${\mathcal D}x^{m-1}={\mathcal B}x^{m-1}+b$. 
Then we have the fixed-point iteration, given implicitly as follows: 
\begin{equation}\label{eqZ}
{\mathcal D}x_{k+1}^{m-1}={\mathcal B}x_{k}^{m-1}+b. 
\end{equation}
(This is the Jacobi iteration when the diagonal elements of ${\mathcal A}$ are positive and ${\mathcal D}$ is the diagonal part of ${\mathcal A}$.) 
Note that ${\mathcal D}x_{k+1}^{m-1}=Dx_{k+1}^{[m-1]}$, where $D$ is the diagonal matrix having the diagonal elements of the tensor ${\mathcal D}$ on the diagonal.  When ${\mathcal B}x_{k}^{m-1}+b\ge 0$, the iteration can be given explicitly as 
$$
x_{k+1}=\left (D^{-1}({\mathcal B}x_{k}^{m-1}+b)\right )^{[1/(m-1)]},
$$ 
but the implicit form will be more convenient for discussions. 

Let $x$ be any element in ${\mathcal S}_g$.
Then $x\ge 0$ and ${\mathcal B}x^{m-1}+b\le {\mathcal D}x^{m-1}$. 
Take $x_0=0$. We can generate a sequence $\{x_k\}$ by iteration \eqref{eqZ}. It is clear that 
$x_0\le x_1$. Suppose $x_{k-1}\le x_k$ ($k\ge 1$). Then 
${\mathcal D}x_{k}^{m-1}= {\mathcal B}x_{k-1}^{m-1}+b\le {\mathcal B}x_{k}^{m-1}+b={\mathcal D}x_{k+1}^{m-1}$. Thus $x_k\le x_{k+1}$. 
Therefore, $x_k\le x_{k+1}$ for all $k\ge 0$. 
Also, $x_0\le x$. Suppose $x_k\le x$ ($k\ge 0$). Then 
 ${\mathcal D}x_{k+1}^{m-1}={\mathcal B}x_{k}^{m-1}+b\le {\mathcal B}x^{m-1}+b\le  {\mathcal D}x^{m-1}$. 
Thus $x_{k+1}\le x$. Therefore, $x_k\le x$ for all $k\ge 0$. Now, $\{x_k\}$ is monotonically increasing and bounded above by $x$. Thus, 
 $\lim_{k\to \infty} x_k=x_{*}$ exists and $x_*\le x$. 
Letting $k\to \infty$ in \eqref{eqZ}, we see that  $x_*$ is a nonnegative solution of \eqref{Teq} and $x_*\le x$ for every $x$ in ${\mathcal S}_g$. 
In particular, $x_*\le x$ for every nonnegative solution $x$ of \eqref{Teq}, so $x_*$ is the minimal nonnegative solution of  \eqref{Teq}, and also the minimal element in ${\mathcal S}_g$. 
\end{proof}

\begin{cor}
Let ${\mathcal A}$ be a nonsingular $M$-tensor and $b\ge 0$. Then equation \eqref{Teq} has a minimal nonnegative solution. 
\end{cor}
\begin{proof}
Since ${\mathcal A}$ be a nonsingular $M$-tensor, we have ${\mathcal A}\hat{x}^{m-1}>0$ for some $\hat{x}>0$. Then for scalar $t>0$ sufficiently large, 
$\tilde{x}=t\hat{x}>0$ is such that 
${\mathcal A}\tilde{x}^{m-1}=t^{m-1}{\mathcal A}\hat{x}^{m-1} \ge b$.   
\end{proof}

The next result is already known in \cite{LGX20}. It also follows quickly from our proof of Theorem \ref{thmmin}. 
\begin{cor}
Let ${\mathcal A}$ be a nonsingular $M$-tensor and $b>0$. Then the unique positive solution of \eqref{Teq} is the only nonnegative solution. 
\end{cor}
\begin{proof}
From our proof of Theorem \ref{thmmin}, we see that the minimal nonnegative solution is positive in this case. 
\end{proof}

\begin{cor}
Let ${\mathcal A}$ be a $Z$-tensor and $b\in {\mathbb R}_+$. Suppose that  $$
{\mathcal S}_g=\{x\in {\mathbb R}_+ \ | \   {\mathcal A}x^{m-1}\ge b\}\ne \emptyset
$$
 and let $x_{\min}({\mathcal A}, b)$ be the minimal nonnegative solution of 
${\mathcal A}x^{m-1}=b$. 
If any element of $b$ decreases but remains nonnegative, or if any diagonal element of ${\mathcal A}$ increases, or if any off-diagonal element of ${\mathcal A}$ increases but remains nonpositive, then the new equation 
$\tilde{{\mathcal A}}x^{m-1}=\tilde{b}$ also has a minimal nonnegative solution $x_{\min}(\tilde{{\mathcal A}}, \tilde{b})$. 
Moreover, $x_{\min}(\tilde{{\mathcal A}}, \tilde{b})\le x_{\min}({\mathcal A}, b)$. 
\end{cor}

\begin{proof}
Under any of those changes, $\tilde{{\mathcal A}}$ is a $Z$-tensor and $\tilde{b}\in {\mathbb R}_+$. 
Let $\tilde{{\mathcal S}}_g=\{x\in {\mathbb R}_+ \ | \   \tilde{{\mathcal A}}x^{m-1}\ge \tilde{b}\}$. 
Then  ${\mathcal S}_g\subseteq \tilde{{\mathcal S}}_g$ and the conclusions follow immediately.  
\end{proof}

The following theorem has been proved in \cite{LGW20}. The proof there is somewhat complicated and does not show how the maximal solution can be computed. 
Here we present a short proof and also a way to compute the maximal solution. 
\begin{thm}\label{MMax}
Let ${\mathcal A}$ be a nonsingular $M$-tensor and
suppose that  ${\mathcal S}_l=\{x\in {\mathbb R}_+ \ | \   {\mathcal A}x^{m-1}\le b\}\ne \emptyset$.  
Then ${\mathcal S}_l$ has a maximal element that is also the maximal nonnegative solution of 
${\mathcal A}x^{m-1}=b$. 
\end{thm}

\begin{proof}
Let ${\mathcal D}$  be the diagonal part of ${\mathcal A}$ and write ${\mathcal A}={\mathcal D}-{\mathcal B}$. 
Then the diagonal elements of ${\mathcal D}$ are positive and ${\mathcal B}\ge 0$. Equation \eqref{Teq} becomes 
${\mathcal D}x^{m-1}={\mathcal B}x^{m-1}+b$. 
Then we have the fixed-point iteration (Jacobi iteration), given implicitly as follows: 
\begin{equation}\label{eqJ}
{\mathcal D}x_{k+1}^{m-1}={\mathcal B}x_{k}^{m-1}+b. 
\end{equation}

Let $x$ be any element in ${\mathcal S}_l$.
Then $x\ge 0$ and ${\mathcal B}x^{m-1}+b\ge {\mathcal D}x^{m-1}$. 
Take $x_0\ge x$ such that ${\mathcal A}x_0^{m-1}\ge b$, so ${\mathcal B}x_0^{m-1}+b\le {\mathcal D}x_0^{m-1}$. 
Note that 
$$
{\mathcal B}x_0^{m-1}+b\ge {\mathcal B}x^{m-1}+b\ge {\mathcal D}x^{m-1}\ge 0. 
$$
Thus $x_1$ is determined by iteration \eqref{eqJ} and ${\mathcal D}x_{1}^{m-1}\ge {\mathcal D}x^{m-1}$, so $x_1\ge x$. 
Also,  
${\mathcal D}x_{1}^{m-1}\le  {\mathcal D}x_{0}^{m-1}$, so $x_1\le  x_0$.  
By induction, we can show that 
$x_{k+1}\le x_{k}$ and $x_k\ge x$ for all $k\ge 0$. Therefore, $\lim_{k\to \infty} x_k=x_{*}$ exists and $x_*\ge x$. 
Then $x_*$ is a nonnegative solution of \eqref{Teq}. 

We still need to show that we can choose one fixed $x_0$ such that ${\mathcal A}x_0^{m-1}\ge b$ and $x_0\ge x$ 
for all $x\in {\mathcal S}_l$. To this end, we take any $\hat{b}>0$ such that $\hat{b}\ge b$, and apply the above argument 
to the equation ${\mathcal A}x^{m-1}=\hat{b}$. We can conclude that  ${\mathcal A}x^{m-1}=\hat{b}$ 
has a nonnegative solution $\hat{x}_*$ (actually the unique positive solution) with $\hat{x}_*\ge x$ for every element $x$ in 
$\{x\in {\mathbb R}_+ \ | \   {\mathcal A}x^{m-1}\le \hat{b}\}$ and thus for every element $x$ in ${\mathcal S}_l$. 

We now return to the equation ${\mathcal A}x^{m-1}=b$ and take $x_0\ge \hat{x}_*$ such that ${\mathcal A}x_0^{m-1}\ge b$. 
Then the sequence $\{x_k\}$ from the Jacobi iteration converges to a nonnegative solution $x_*$ of \eqref{Teq} and  
$x_*\ge x$ for all elements $x$ in ${\mathcal S}_l$. 
This $x_*$ is then the maximal element in ${\mathcal S}_l$ and also the maximal nonnegative solution of  \eqref{Teq}. 
\end{proof}

\begin{remark}
From the proof, we see that the maximal nonnegative solution (when exists) can be found by iteration \eqref{eqJ} using any $x_0$ such that 
$$
x_0>0, \quad {\mathcal A}x_0^{m-1}>0,\quad {\mathcal A}x_0^{m-1}\ge b. 
$$
In other words, we can take $x_0$ to be the unique positive solution of ${\mathcal A}x_0^{m-1}=\hat{b}$, where $\hat{b}$ is any vector such that 
$\hat{b}>0$ and $\hat{b}\ge b$. 
\end{remark}

\begin{remark}
If ${\mathcal A}$ is a $Z$-tensor, but not a nonsingular $M$-tensor, then equation \eqref{Teq} may not have a maximal nonnegative solution when it has nonnegative solutions.  One example is the equation \eqref{Teq} with ${\mathcal A}\in {\mathbb R}^{[4,3]}$ given by: 
$$
a_{1111}=0, a_{2222}=a_{3333}=1, a_{1112}=a_{3111}=-1, \mbox{ and } a_{i_1i_2i_3i_4}=0 \mbox{ elsewhere,}
$$
and $b=[0,0,1]^T$. The equation has infinitely many nonnegative solutions, given by $[c, 0, (1+c^3)^{1/3}]^T$ for any $c\ge 0$. 
The minimal nonnegative solution is $[0, 0, 1]^T$, but the maximal nonnegative solution does not exist. 
\end{remark}

\begin{cor}
Let ${\mathcal A}$ be a nonsingular $M$-tensor and $b\ge 0$. Then equation \eqref{Teq} has a maximal nonnegative solution. 
\end{cor}

\begin{remark}
When ${\mathcal A}$ is a nonsingular $M$-tensor and $b\ge 0$, we can use Jacobi iteration with 
$x_0=0$ to get the minimal solution, and use $x_0>0$ with ${\mathcal A}x_0^{m-1}>0$  and ${\mathcal A}x_0^{m-1}\ge b$  to get the maximal solution. 
When $b>0$, we can use either of these two choices to get the unique positive solution. 
For the case $b>0$, the Jacobi iteration has been studied in \cite{DW16} with $x_0>0$ satisfying 
$0<{\mathcal A}x_0^{m-1}\le b$. More general tensor splitting methods have been studied in \cite{LLV18}, again with this requirement on $x_0$ (see \cite[Theorem 5.4]{LLV18}). 
It is interesting to note that both our choices are excluded by this requirement (unless $x_0$ is already the solution). 
\end{remark}

\begin{cor}\label{corcomp}
Let ${\mathcal A}$ be a nonsingular $M$-tensor. 
Suppose that  
$$
{\mathcal S}_l=\{x\in {\mathbb R}_+ \ | \   {\mathcal A}x^{m-1}\le b\}\ne \emptyset
$$
 and let $x_{\max}({\mathcal A}, b)$ be the maximal nonnegative solution of 
${\mathcal A}x^{m-1}=b$. 
If any element of $b$ increases, or if any entry of ${\mathcal A}$ decreases, then the new equation 
$\tilde{{\mathcal A}}x^{m-1}=\tilde{b}$ also has a maximal nonnegative solution $x_{\max}(\tilde{{\mathcal A}}, \tilde{b})$ 
provided that $\tilde{{\mathcal A}}$ is still a nonsingular $M$-tensor. 
Moreover, $x_{\max}(\tilde{{\mathcal A}}, \tilde{b})\ge x_{\max}({\mathcal A}, b)$. 
\end{cor}

\begin{proof}
Let $\tilde{{\mathcal S}}_l=\{x\in {\mathbb R}_+ \ | \   \tilde{{\mathcal A}}x^{m-1}\le \tilde{b}\}$. 
Then  ${\mathcal S}_g\subseteq \tilde{{\mathcal S}}_g$ and the conclusions follow immediately.  
\end{proof}

The next result partially explain why the maximal nonnegative solution  is of particular interest for equation \eqref{Teq} with $b\ge 0$. 
\begin{cor}
Let ${\mathcal A}$ be a nonsingular $M$-tensor and $b\ge 0$. Suppose $x^{(k)}$ is the unique positive solution of 
${\mathcal A}x^{m-1}=b^{(k)}$, where $b^{(k)}>0$ ($k=1, 2, \ldots$) and $b^{(k)}$ is monotonically decreasing and converges to $b$ as $k\to \infty$. 
Then, as $k\to \infty$, $x^{(k)}$ converges to the maximal nonnegative solution $x_{\max}({\mathcal A}, b)$ of  ${\mathcal A}x^{m-1}=b$. 
\end{cor}

\begin{proof} 
By Corollary \ref{corcomp}, $x^{(k)} \ge x^{(k+1)}\ge x_{\max}({\mathcal A}, b)$ for all $k\ge 0$, 
so $\lim_{k\to \infty} x^{(k)}=x_*$ exists. 
Letting  $k \to \infty$ in ${\mathcal A}(x^{(k)})^{m-1}=b^{(k)}$, we get ${\mathcal A}x_*^{m-1}=b$ and 
$x_*\ge x_{\max}({\mathcal A}, b)$. Thus $x_*= x_{\max}({\mathcal A}, b)$.  
\end{proof}

\section{Iterative methods for extremal nonnegative solutions}
\label{sec3}

In the previous section, we have presented new proofs for the existence of extremal nonnegative solutions of the tensor equation 
${\mathcal A}x^{m-1}=b$, where ${\mathcal A}$ is a nonsingular $M$-tensor and $b$ is a nonnegative vector, by using the Jacobi iteration with suitable initial guesses. 
We have also presented some results when ${\mathcal A}$ is a $Z$-tensor or $b$ is a general real vector. 
Now, we would like to present some iterative methods that may be more efficient than the Jacobi iteration for actual computation of the extremal solutions, determine and compare the rates of convergence of these methods. 

The Jacobi iteration is just a very simple fixed-point iteration. In the Jacobi iteration, we keep the term involving $x_i^{m-1}$ in the $i$th equation
($i\in [n]$) on the left and move all other terms to the right. But in the $i$th equation, we may also have terms $x_j^{m-1}$ with $j\ne i$. 
We may consider keeping all unmixed terms ($x_1^{m-1}, \ldots, x_n^{m-1}$) on the left, regardless which equation they are from. 
In other words, we have a splitting of the tensor ${\mathcal A}$: ${\mathcal A}={\mathcal M}-{\mathcal N}$, where ${\mathcal M}=(m_{i_1i_2\ldots i_m})$ with 
$m_{ij\ldots j}=a_{ij\ldots j}$  for $i, j \in [n]$ and $m_{i_1i_2\ldots i_m} =0$ elsewhere. We may call this splitting a level-1 splitting. 
If we keep all unmixed terms on the left, we then have the equation 
$$
{\mathcal M}x^{m-1}={\mathcal N}x^{m-1}+b. 
$$
Note that  ${\mathcal M}x^{m-1}=Mx^{[m-1]}$, where the $n\times n$ matrix $M$ has $(i,j)$ element $a_{ij...j}$  for $i, j \in [n]$, 
and is called \cite{P10} the majorization matrix associated with ${\mathcal A}$. 
It is easy to see \cite{LLV18} that $M$  is a nonsingular $M$-matrix when ${\mathcal  A}$  is a nonsingular $M$-tensor. 
When ${\mathcal A}$ is a $Z$-tensor, $M$ is obviously a $Z$-matrix. 

We then have a splitting of the matrix $M$: $M=P-Q$, where $P$ is a nonsingular $M$-matrix and $Q\ge 0$. This may be called a level-2 splitting. 
For example, the following splitting of a $Z$-matrix is permitted, although not a good one. 
$$
\left [\begin{array}{ccc}
-1 & 0 & 0\\
-3 & 2 & -2\\
0 & -3 & 4
\end{array} \right ]=\left [\begin{array}{ccc}
2 & 0 & 0\\
-1 & 3 & -1\\
0 & -2 & 5
\end{array} \right ]-\left [\begin{array}{ccc}
3 & 0 & 0\\
2 & 1 & 1\\
0 & 1& 1
\end{array} \right ]. 
$$
We can now rewrite ${\mathcal A}x^{m-1}=b$ as
$$
Px^{[m-1]}=Qx^{[m-1]}+{\mathcal N}x^{m-1}+b
$$
and get fixed-point iteration in implicit form
\begin{equation}\label{eqM}
Px_{k+1}^{[m-1]}=Qx_k^{[m-1]}+{\mathcal N}x_k^{m-1}+b
\end{equation}
or in explicit form
$$
x_{k+1}=\left (P^{-1}\left (Qx_k^{[m-1]}+{\mathcal N}x_k^{m-1}+b\right )\right )^{[1/(m-1)]}, 
$$
assuming  $P^{-1}\left (Qx_k^{[m-1]}+{\mathcal N}x_k^{m-1}+b\right )\ge 0$ for each $k$. 
In \cite{LLV18}, fixed-point iterations like this are called tensor splitting iterative methods and studied for finding the unique positive solution of the tensor equation  ${\mathcal A}x^{m-1}=b$ with ${\mathcal A}$ being a nonsingular $M$-tensor and $b$ a positive vector. 

We first study iteration \eqref{eqM} for finding the minimal nonnegative solution. 

\begin{thm}\label{thmminM}
Let ${\mathcal A}$ be a $Z$-tensor and $b\in {\mathbb R}_+$. Suppose that  $$
{\mathcal S}_g=\{x\in {\mathbb R}_+ \ | \   {\mathcal A}x^{m-1}\ge b\}\ne \emptyset. 
$$ 
Then for $x_0=0$, the sequence $\{x_k\}$ from iteration \eqref{eqM} is monotonically increasing and converges to the minimal nonnegative solution of 
${\mathcal A}x^{m-1}=b$. 
\end{thm}

\begin{proof}
The proof is almost the same as the proof of Theorem \ref{thmmin}. In that proof, we use the obvious fact that 
${\mathcal D}x^{m-1}\ge {\mathcal D}y^{m-1}$ (with $x, y\ge 0$) implies $x\ge y$. 
We now need ${P}x^{[m-1]}\ge {P}y^{[m-1]}$  implies $x\ge y$, which is true since $P^{-1}\ge 0$ for the nonsingular $M$-matrix $P$.  
\end{proof}

\begin{remark}
When ${\mathcal A}$ is a $Z$-tensor but not a nonsingular $M$-tensor, it may not be easy to determine whether ${\mathcal S}_g \ne \emptyset$.   
We may apply iteration \eqref{eqM} without checking this condition. In this case, the sequence $\{x_k\}$ is still well-defined and monotonically increasing. 
The sequence is bounded above if and only if ${\mathcal S}_g \ne \emptyset$.  
\end{remark}

In iteration \eqref{eqM}, ${\mathcal N}$ is uniquely determined by ${\mathcal A}$, but we have the freedom to choose $Q$ ($P$ is uniquely determined by $Q$). 
The next result gives a comparison of convergence rate by examining the iterates right from the beginning (so we are not talking about asymptotic rate of convergence here). 

\begin{thm}\label{ratecompm} 
Under the conditions of Theorem \ref{thmminM}, 
let $\{x_k\}$  and $\{\hat{x}_k\}$ be the sequences generated by iteration \eqref{eqM} with two splittings of  $M$: 
$M=P-Q$ and $M=\hat{P}-\hat{Q}$, respectively, and with $x_0= \hat{x}_0=0$. 
If $\hat{Q}\le Q$,  then $x_k\le \hat{x}_k$ for all $k\ge 0$. In other words, a smaller matrix $Q$ gives faster termwise convergence 
for finding the minimal nonnegative  solution. 
\end{thm} 

\begin{proof} We need to show that $x_k\le \hat{x}_k$ implies $x_{k+1}\le \hat{x}_{k+1}$  for each $k\ge 0$. 
We have \eqref{eqM} and 
\begin{equation}\label{eqhat}
\hat{P}\hat{x}_{k+1}^{[m-1]}=\hat{Q}\hat{x}_k^{[m-1]}+{\mathcal N}\hat{x}_k^{m-1}+b. 
\end{equation}
From $P-Q=\hat{P}-\hat{Q}$, we get $\hat{P}={P}-(Q-\hat{Q})$ and get from \eqref{eqhat} that 
\begin{eqnarray*}
P\hat{x}_{k+1}^{[m-1]}&=&(Q-\hat{Q})\hat{x}_{k+1}^{[m-1]}+\hat{Q}\hat{x}_k^{[m-1]}+{\mathcal N}\hat{x}_k^{m-1}+b\\
&\ge&(Q-\hat{Q})\hat{x}_{k}^{[m-1]}+\hat{Q}\hat{x}_k^{[m-1]}+{\mathcal N}\hat{x}_k^{m-1}+b\\
&=& Q\hat{x}_{k}^{[m-1]}+{\mathcal N}\hat{x}_k^{m-1}+b\\
&\ge & Q x_{k}^{[m-1]}+{\mathcal N}{x}_k^{m-1}+b\\
&=&Px_{k+1}^{[m-1]}. 
\end{eqnarray*}
Thus  $x_{k+1}\le \hat{x}_{k+1}$.  
\end{proof}

We now show that iteration \eqref{eqM} has linear convergence under suitable assumptions. By linear convergence we actually mean at least linear convergence. For example, $x_0=0$ is already a solution when $b=0$. Let ${\mathcal L}$ be the nonnegative tensor such that ${\mathcal L}x^{m-1} = Qx^{[m-1]}+{\mathcal N}x^{m-1}$. 
Then iteration \eqref{eqM} becomes 
\begin{equation}\label{eqL}
Px_{k+1}^{[m-1]}={\mathcal L}x_k^{m-1}+b.
\end{equation}

For $x\in {\mathbb R}^n$ and index set $I\subseteq  [n]$, we denote by $x_I$ the subvector of $x$, whose
elements are $x_i$, $i\in I$.  For tensor ${\mathcal A} = (a_{i_1\ldots i_m}) \in {\mathbb R}^{[m, n]}$, we denote by ${\mathcal A}_I$  the 
subtensor of ${\mathcal A}$  with elements $a_{i_1\ldots i_m}$,  $i_1, \ldots, i_m\in  I$. 

\begin{thm}\label{rateMin}
Let ${\mathcal A}$ be a $Z$-tensor and $b\in {\mathbb R}_+$. Let $I_0=\{i \ | \ b_i=0\}$ 
and suppose that  ${\mathcal S}_g=\{x\in {\mathbb R}_+ \ | \   {\mathcal A}x^{m-1}\ge b\}\ne \emptyset$.  
Let $\{x_k\}$ be  the sequence  from iteration \eqref{eqL} with $x_0=0$. 
\begin{enumerate}
\item If $I_0=[n]$, then $x_0=0$ is already the minimal nonnegative solution. 
\item If $I_0=\emptyset$, then $\{x_k\}$ converges linearly to the minimal nonnegative solution 
(which is actually the unique positive solution) of ${\mathcal A}x^{m-1}=b$. 
\item If $I_0$ is a proper subset of $[n]$ and $k_0$ is the smallest integer such that $x_{k_0}$ and $x_{k_0+1}$ have the same zero pattern, then $1\le k_0\le n$ and $x_k$ have the same zero pattern for all $k\ge k_0$, which is also the zero pattern of the minimal nonnegative solution $x_{\min}$. Let $I=\{i \ | \ (x_{k_0})_i=0\}$. Then $I \subseteq I_0$. Let $I_c=[n]\setminus I$.  
Then the iteration \eqref{eqL} for $k\ge k_0$ is reduced to an iteration for the lower-dimensional tensor equation:  
\begin{equation}\label{eqred}
{\mathcal A}_{I_c}\hat{x}^{m-1}=b_{I_c}. 
\end{equation}
For $k\ge k_0$, $\hat{x}_k$ from the reduced iteration is the same as $(x_{k})_{I_c}$.
The minimal nonnegative solution $\hat{x}_{\min}$ of \eqref{eqred} is positive and is the same as $(x_{\min})_{I_c}$. 
Thus  $x_{k}$ converges to $x_{\min}$ linearly if and only if $\hat{x}_{k}$ converges to $\hat{x}_{\min}$ linearly. 
Assume ${\mathcal A}{x}^{m-1}=b$ is already the reduced equation for notation convenience (all diagonal elements of ${\mathcal A}$ are now positive). Then $x_{k}$ converges to $x_{\min}$ linearly under any of the following four conditions: 
\begin{enumerate}
\item $P^{-1}b>0$.
\item $P^{-1}\bar{\mathcal L} e^{m-2}$ is irreducible, where $e$ is the vector of ones and $\bar{\mathcal L}$ is the semi-symmetric tensor from ${\mathcal L}$. 
\item With all diagonal elements of ${\mathcal A}$ included entirely in $P$ and ${\mathcal B}={\mathcal D}-{\mathcal A}$ (${\mathcal D}$ is the diagonal part of ${\mathcal A}$), the 
matrix $\bar{\mathcal B} e^{m-2}$ is strictly upper or lower triangular.
\item For each $i\in I_0$, there is an element $a_{ii_2\ldots i_m}\ne 0$, where $i_j\notin I_0$ for at least one $j$ ($2\le j\le m$). 
\end{enumerate}
\end{enumerate}
\end{thm} 

\begin{proof} 
Conclusion 1 in the theorem is obvious. The proof of conclusion 2 is the same as the proof of conclusion 3 for the reduced equation under condition (a).  
We now prove conclusion 3. 

With $x_0=0$ for iteration \eqref{eqL}, we have $x_k\le x_{k+1}$ for each $k\ge 0$. Let $k_0$ be the smallest integer such that $x_{k_0}$ and $x_{k_0+1}$ have the same zero pattern. It is clear that $1\le k_0\le n$ and that the zero pattern will not change afterwards and the minimal solution $x_{\min}$  has the same zero pattern. Let $I=\{i \ | \ (x_{k_0})_i=0\}$. Since $b_i>0$ implies 
$(x_{k_0})_i\ge (x_1)_i>0$ for a diagonal $P$, we have $I \subseteq  I_0$. For $k\ge k_0$, we have $(x_k)_I=0$. 
Let $I_c=[n]\setminus I$.   
Since $(x_k)_{i_2}(x_k)_{i_3}\cdots (x_k)_{i_m}=0$ for all $k\ge k_0$ when $i_j\in I$ for at least one $j$ 
($j=2, \ldots, m$), we only need the elements in ${\mathcal A}_{I_c}$ to continue the iteration \eqref{eqL} for $k\ge k_0$. 
In other words, the iteration \eqref{eqL} for $k\ge k_0$ is reduced to an iteration for the lower-dimensional tensor equation with indices running through $I_c$ only:  
\begin{equation}\label{eqred2}
{\mathcal A}_{I_c}\hat{x}^{m-1}=b_{I_c}, 
\end{equation}
with $\hat{x}_{k_0}$ obtained from $x_{k_0}$ by deleting all zero elements. 
Note that the level-1 and level-2 splittings of ${\mathcal A}_{I_c}$ are the ones inherited from those for ${\mathcal A}$.
The minimal nonnegative solution $\hat{x}_{\min}$ of \eqref{eqred2} is positive (obtained from $x_{\min}$ by deleting all zero elements), 
but $b_{I_c}$ will have some zero elements when $I\ne I_0$. 
For $k\ge k_0$, $\hat{x}_k$ from the reduced iteration can also be obtained from $x_{k}$ by deleting all zero elements. 
Thus  $x_{k}$ converges to $x_{\min}$ linearly if and only if $\hat{x}_{k}$ converges to $\hat{x}_{\min}$ linearly. 
We now assume that ${\mathcal A}{x}^{m-1}=b$ is already the reduced equation for notation convenience, and denote its positive minimal solution by $\bar{x}$ for short. Note that all diagonal elements of the reduced tensor are positive. 

By Theorem \ref{ratecompm}, we may assume $P$ is a diagonal matrix (where slower convergence happens). 
So we will assume $P$ is diagonal when needed.

The iteration \eqref{eqL} can be written explicitly as $x_{k+1}=\phi(x_k)$, with 
$$
\phi(x)=\left (P^{-1}({\mathcal L}x^{m-1}+b)\right )^{[1/(m-1)]}. 
$$
From 
$$
P\phi(x)^{[m-1]}={\mathcal L}x^{m-1}+b=\bar{\mathcal L}x^{m-1}+b, 
$$
where $\bar{\mathcal L}$ is the semi-symmetric tensor obtained from ${\mathcal L}$, we 
take derivative on both sides to obtain 
$$
P (m-1){\rm diag} (\phi(x)^{[m-2]})\phi'(x)=(m-1)\bar{\mathcal L}x^{m-2}, 
$$
where for $x\in {\mathbb R}^n$, ${\rm diag} (x)$ is the diagonal matrix with the elements of $x$ on the diagonal. 
We need to show $\rho(\phi'(\bar{x}))<1$ for 
\begin{equation}\label{phiex}
\phi'(\bar{x})= {\rm diag} (\bar{x}^{[-(m-2)]})P^{-1}\bar{\mathcal L}{\bar{x}}^{m-2}.
\end{equation} 
Note that the nonnegative matrix $\phi'(\bar{x})$ is such that 
\begin{eqnarray*}
\phi'(\bar{x})\bar{x}&=& {\rm diag} (\bar{x}^{[-(m-2)]})P^{-1}\bar{\mathcal L}{\bar{x}}^{m-2}\bar{x}\\
&=&{\rm diag} (\bar{x}^{[-(m-2)]})P^{-1}\bar{\mathcal L}{\bar{x}}^{m-1}\\
&=&{\rm diag} (\bar{x}^{[-(m-2)]})P^{-1}{\mathcal L}{\bar{x}}^{m-1}\\
&=&{\rm diag} (\bar{x}^{[-(m-2)]})P^{-1}(P{\bar{x}}^{[m-1]}
-{\mathcal A}{\bar{x}}^{m-1})\\
&=&\bar{x}-{\rm diag} (\bar{x}^{[-(m-2)]})P^{-1}b.
\end{eqnarray*}
If $P^{-1}b>0$, then we have $\phi'(\bar{x})\bar{x}<\bar{x}$ and thus \cite{BP94} $\rho(\phi'(\bar{x}))<1$. 
This shows that condition (a) is sufficient for linear convergence. 
Conditions (b), (c) and (d) are needed only when $I_0\ne \emptyset$ for the reduced equation. 

Since $P^{-1}b\ge 0$ and $P^{-1}b\ne 0$, 
we have $\phi'(\bar{x})\bar{x}\le \bar{x}$ and $\phi'(\bar{x})\bar{x}\ne \bar{x}$.  Thus we still have \cite{BP94}  $\rho(\phi'(\bar{x}))<1$ if  $\phi'(\bar{x})$ 
is irreducible. From \eqref{phiex}, we see that $\phi'(\bar{x})$ is irreducible if and only if $P^{-1}\bar{\mathcal L}e^{m-2}$ is irreducible. 
Thus, condition (b) is also sufficient for linear convergence. 

Under condition (c), we just need to show linear convergence for the Jacobi iteration. Now, 
$$
\phi'(\bar{x})= {\rm diag} (\bar{x}^{[-(m-2)]})D^{-1}\bar{\mathcal B}{\bar{x}}^{m-2}.
$$
Since $\bar{\mathcal B} e^{m-2}$ is strictly upper or lower triangular, so is $\phi'(\bar{x})$. Thus $\rho(\phi'(\bar{x}))=0$ and linear convergence follows. 

Finally, we show that condition (d) is also sufficient for linear convergence. We now assume $P$ is diagonal. 
We modify $\bar{x}$ to $\hat{x}$ by changing $(\bar{x})_i$ to $(\bar{x})_i-\epsilon>0$  for each  $i\notin I_0$, where $0<\epsilon< 
\min_{i \notin I_0}(\bar{x})_i^{-(m-2)}(P^{-1}b)_i$. 
Now for $i\notin I_0$, 
$$(\phi'(\bar{x})\hat{x})_i\le (\phi'(\bar{x})\bar{x})_i=(\bar{x})_i-(\bar{x})_i^{-(m-2)}(P^{-1}b)_i<(\hat{x})_i. $$
For $i \in I_0$, 
$$(\phi'(\bar{x})\hat{x})_i\le (\phi'(\bar{x})\bar{x})_i=(\bar{x})_i=(\hat{x})_i. $$ 
Since $P$ is diagonal, $(\phi'(\bar{x})\hat{x})_i = (\phi'(\bar{x})\bar{x})_i$
if and only if $(\bar{\mathcal L}{\bar{x}}^{m-2}\hat{x})_i = (\bar{\mathcal L}{\bar{x}}^{m-2}\bar{x})_i$, i.e., 
\begin{eqnarray*}
&&\sum_{j=1}^n\left (\sum_{i_3, \ldots, i_m=1}^n \bar{\mathcal L}_{iji_3\ldots i_m}\bar{x}_{i_3}\cdots \bar{x}_{i_m}\right )\hat{x}_j\\
&=&\sum_{j=1}^n\left (\sum_{i_3, \ldots, i_m=1}^n \bar{\mathcal L}_{iji_3\ldots i_m}\bar{x}_{i_3}\cdots \bar{x}_{i_m}\right )\bar{x}_j. 
\end{eqnarray*}
This holds if and only if $\bar{\mathcal L}_{iji_3\ldots i_m}=0$ for all $j\notin I_0$ and for all $i_3, \ldots, i_m\in [n]$, 
i.e., ${\mathcal L}_{ii_2i_3\ldots i_m}=0$ for all $i_2, \ldots, i_m\in [n]$ such that $i_j\notin I_0$ for at least one $j$ ($2\le j\le m$). 
Therefore, when condition (d) holds, $(\phi'(\bar{x})\hat{x})_i < (\phi'(\bar{x})\bar{x})_i=(\hat{x})_i$ for each $i\in I_0$. 
Linear convergence follows since we again have $\rho(\phi'(\bar{x}))<1$. 
 \end{proof}

We now study iteration \eqref{eqM} for finding the maximal nonnegative solution. 

\begin{thm}\label{MMaxCon}
Let ${\mathcal A}$ be a nonsingular $M$-tensor and
suppose that  ${\mathcal S}_l=\{x\in {\mathbb R}_+ \ | \   {\mathcal A}x^{m-1}\le b\}\ne \emptyset$.  Then for any 
$x_0>0$ such that ${\mathcal A}x_0^{m-1}> 0$ and ${\mathcal A}x_0^{m-1}\ge b$. 
The sequence from iteration \eqref{eqM} is monotonically decreasing and converges to the maximal nonnegative solution of 
${\mathcal A}x^{m-1}=b$. 
\end{thm}

\begin{proof}
The proof is almost the same as the proof of Theorem \ref{MMax}. In that proof, we use the fact that 
${\mathcal D}x^{m-1}\ge {\mathcal D}y^{m-1}$ (with $x, y\ge 0$) implies $x\ge y$. 
We now use the fact ${P}x^{[m-1]}\ge {P}y^{[m-1]}$  implies $x\ge y$.   
\end{proof}

\begin{remark}
When $b\notin {\mathbb R}^+$, it may not be easy to determine whether ${\mathcal S}_l\ne \emptyset$. 
We may apply iteration \eqref{eqM} without checking this condition. In this case, ${\mathcal S}_l\ne \emptyset$ 
if and only if $P^{-1}\left (Qx_k^{[m-1]}+{\mathcal N}x_k^{m-1}+b\right )\ge 0$ for each $k\ge 0$. 
Indeed, if $P^{-1}\left (Qx_k^{[m-1]}+{\mathcal N}x_k^{m-1}+b\right )\ge 0$ for each $k\ge 0$, then we see from a proof similar to that of Theorem 
\ref{MMax} that $x_k\ge x_{k+1}\ge 0$ for all $k\ge 0$. In this case, $\lim_{k\to \infty}x_k=x_*$ exists and $x_*$ is a solution of ${\mathcal A}x^{m-1}=b$
 and thus 
${\mathcal S}_l\ne \emptyset$. 
\end{remark}

\begin{remark}
To obtain a suitable $x_0$ in Theorem \ref{MMaxCon}, we can always take a vector $\hat{b}>0$ with $\hat{b}\ge b$, and use the methods in 
\cite{DW16,H17,HLQZ18} to get the unique positive solution $x_*$ of the equation 
${\mathcal A}x^{m-1}=\hat{b}$ and then take $x_0=x_*$. 
When we need to solve the tensor equation with the same ${\mathcal A}$ and many different  right side vectors $b_1, \ldots, b_p$, we can take a vector  $\hat{b}>0$ with $\hat{b}\ge \max\{b_1,\ldots, b_p\}$ (where the maximum is taken elementwise) and get the unique positive solution $x_*$ 
of the equation ${\mathcal A}x^{m-1}=\hat{b}$ and then take $x_0=x_*$ for all equations 
${\mathcal A}x^{m-1}={b}_i$, $i=1,\ldots, p$. 
\end{remark}

\begin{thm}\label{ratecompM}
Under the conditions of Theorem \ref{MMaxCon}, 
let $\{x_k\}$  and $\{\hat{x}_k\}$ be the sequences generated by iteration \eqref{eqM} with two splittings of  $M$: 
$M=P-Q$ and $M=\hat{P}-\hat{Q}$, respectively, and with $x_0\ge \hat{x}_0>0$ satisfying 
$$
{\mathcal A}x_0^{m-1}>0, \quad  {\mathcal A}x_0^{m-1}\ge b,  \quad {\mathcal A}\hat{x}_0^{m-1}>0, \quad {\mathcal A}\hat{x}_0^{m-1}\ge b.
$$
If $\hat{Q}\le Q$,  then $x_k\ge \hat{x}_k$ for all $k\ge 0$. In other words, a smaller matrix $Q$ gives faster termwise convergence 
for finding the maximal solution. 
\end{thm} 

\begin{proof} We need to show that $x_k\ge \hat{x}_k$ implies $x_{k+1}\ge \hat{x}_{k+1}$  for each $k\ge 0$. 
We have  \eqref{eqM} and \eqref{eqhat}. 
Since $\hat{P}={P}-(Q-\hat{Q})$, we get from \eqref{eqhat} that 
\begin{eqnarray*}
P\hat{x}_{k+1}^{[m-1]}&=&(Q-\hat{Q})\hat{x}_{k+1}^{[m-1]}+\hat{Q}\hat{x}_k^{[m-1]}+{\mathcal N}\hat{x}_k^{m-1}+b\\
&\le&(Q-\hat{Q})\hat{x}_{k}^{[m-1]}+\hat{Q}\hat{x}_k^{[m-1]}+{\mathcal N}\hat{x}_k^{m-1}+b\\
&=& Q\hat{x}_{k}^{[m-1]}+{\mathcal N}\hat{x}_k^{m-1}+b\\
&\le & Q x_{k}^{[m-1]}+{\mathcal N}{x}_k^{m-1}+b\\
&=&Px_{k+1}^{[m-1]}. 
\end{eqnarray*}
Thus  $x_{k+1}\ge \hat{x}_{k+1}$.  
\end{proof}

We now study the convergence rate of iteration \eqref{eqM} for finding the maximal nonnegative solution of \eqref{Teq}, under the conditions of Theorem \ref{MMaxCon}. In our first result we assume that the maximal nonnegative solution is positive. We know from \cite[Theorem 2.4]{LGX20} that  
every nonnegative solution of \eqref{Teq} is positive if ${\mathcal A}$ is irreducible and $b\ge 0$ is nonzero. 
We also know from Corollary \ref{corcomp} that if the maximal nonnegative solution of ${\mathcal A}x^{m-1}=b$ is positive, then 
the maximal nonnegative solution of $\tilde{{\mathcal A}}x^{m-1}=\tilde{b}$ will also be positive if $\tilde{{\mathcal A}}$ and $\tilde{b}$ are obtained from 
${\mathcal A}$ and $b$ in ways described there.

\begin{thm}\label{rateMaxP}
Let ${\mathcal A}$ be a nonsingular $M$-tensor and suppose that  ${\mathcal S}_l=\{x\in {\mathbb R}_+ \ | \   {\mathcal A}x^{m-1}\le b\}\ne \emptyset$.  
Let $x_{\max}$ be the maximal nonnegative solution of ${\mathcal A}x^{m-1}=b$ and assume that $x_{\max}>0$.  Let $I_0=\{i \ | \ b_i=0\}$. 
Then for any $x_0>0$ such that  ${\mathcal A}x_0^{m-1}> 0$ and ${\mathcal A}x_0^{m-1}\ge b$,  
the sequence  $\{x_k\}$ from iteration \eqref{eqM} converges to $x_{max}$ linearly 
under any of the following four conditions: 
\begin{enumerate}
\item $P^{-1}b>0$.
\item $P^{-1}\bar{\mathcal L} e^{m-2}$ is irreducible, where $e$ is the vector of ones and $\bar{\mathcal L}$ is the semi-symmetric tensor from ${\mathcal L}$. 
\item With all diagonal elements of ${\mathcal A}$ included entirely in $P$ and ${\mathcal B}={\mathcal D}-{\mathcal A}$, the 
matrix $\bar{\mathcal B} e^{m-2}$ is strictly upper or lower triangular.
\item For each $i\in I_0$, there is an element $a_{ii_2\ldots i_m}\ne 0$, where $i_j\notin I_0$ for at least one $j$ ($2\le j\le m$). 
\end{enumerate}
\end{thm}

\begin{proof}
When $b=0$, ${\mathcal A}x^{m-1}=b$ has a unique solution $x=0$. Thus $b\ne 0$ when $x_{\max}>0$. 
Exactly as in the proof of Theorem \ref{rateMin} (we do not need the reduction process there since we assume $x_{\max}>0$), we can show that the iteration map $\phi$ is such that 
$\rho(\phi'(x_{\max}))<1$ under any of the four conditions in the theorem. In proving linear convergence under condition (c) or (d), 
we may assume that $P$ is diagonal. 
\end{proof}

We now assume $b\ge 0$, but allow $x_{\max}$ to have some zero elements. 

\begin{thm}\label{rateMax}
Let ${\mathcal A}$ be a nonsingular $M$-tensor and $b\in {\mathbb R}_+$ be nonzero. 
Let $x_{\max}$ be the maximal nonnegative solution of ${\mathcal A}x^{m-1}=b$.  Let $I_0=\{i \ | \ b_i=0\}$. 
Then for iteration \eqref{eqM} with any $x_0>0$ such that ${\mathcal A}x_0^{m-1}> 0$ and ${\mathcal A}x_0^{m-1}\ge b$,  
there is a smallest integer $k_0$ ($0\le k_0\le n-1$) such that $x_{k_0}$ and $x_{k_0+1}$ have the same zero pattern (including the case with no zero elements), and $x_k$ have the same zero pattern for all $k\ge k_0$. Let $I=\{i \ | \ (x_{k_0})_i=0\}$ (which may be empty). Then $I \subseteq I_0$. Let $I_c=[n]\setminus I$. 
Then the iteration \eqref{eqM} for $k\ge k_0$ is reduced to an iteration for the lower-dimensional tensor equation:  
\begin{equation}\label{eqredM}
{\mathcal A}_{I_c}\hat{x}^{m-1}=b_{I_c}. 
\end{equation}
For $k\ge k_0$, $\hat{x}_k$ from the reduced iteration is the same as $(x_{k})_{I_c}$.
The maximal nonnegative solution $\hat{x}_{\max}$ of \eqref{eqredM} is the same as $(x_{\max})_{I_c}$ (which is not necessarily positive).  
Thus  $x_{k}$ converges to $x_{\max}$ linearly if and only if $\hat{x}_{k}$ converges to $\hat{x}_{\max}$ linearly. 
Assume ${\mathcal A}{x}^{m-1}=b$ is already the reduced equation and assume that its maximal solution $x_{\max}$ is positive. Then $x_{k}$ converges to $x_{\max}$ linearly under any of the following four conditions: 
\begin{enumerate}
\item $P^{-1}b>0$.
\item $P^{-1}\bar{\mathcal L} e^{m-2}$ is irreducible, where $e$ is the vector of ones and $\bar{\mathcal L}$ is the semi-symmetric tensor from ${\mathcal L}$. 
\item With all diagonal elements of ${\mathcal A}$ included entirely in $P$ and ${\mathcal B}={\mathcal D}-{\mathcal A}$, the 
matrix $\bar{\mathcal B} e^{m-2}$ is strictly upper or lower triangular.
\item For each $i\in I_0$, there is an element $a_{ii_2\ldots i_m}\ne 0$, where $i_j\notin I_0$ for at least one $j$ ($2\le j\le m$). 
\end{enumerate}
\end{thm} 

\begin{proof}
We have $x_0>0$ and $x_k\ge x_{k+1}$ for all $k\ge 0$. Each $x_k$ has at least one nonzero entry since $x_{\max}\ne 0$. 
Thus, there is a smallest integer $k_0$ ($0\le k_0\le n-1$) such that $x_{k_0}$ and $x_{k_0+1}$ have the same zero pattern. 
Since $b\ge 0$, $x_k$ have the same zero pattern for all $k\ge k_0$. It is clear that  $I \subseteq I_0$. The reduced equation \eqref{eqredM} is then obtained as in the proof of Theorem \ref{rateMin}. However, its maximal solution $x_{\max}$ (as the limit of a positive sequence) may have some zero elements. With the additional assumption that $x_{\max}>0$ for the reduced equation, the proof of linear convergence is exactly the same as in the proof of Theorem \ref{rateMin}.  
\end{proof}

\begin{remark}
The reduction in Theorem \ref{rateMin} is obtained from the iteration for finding the minimal solution and the reduction in Theorem \ref{rateMax} is obtained from the iteration for finding the maximal solution. A reduction process has also been described in \cite{LGX20}, without mentioning maximal and minimal solutions. We will explain that one can only find a nonnegative solution with the same number of nonzero elements as the minimal solution by finding a positive solution of the reduced equation in \cite{LGX20}. 
\end{remark}

To describe the approach in \cite{LGX20}, we recall the following definition. 

\begin{definition} 
A tensor ${\mathcal A}\in {\mathbb R}^{[m,n]}$  is called reducible with respect to $I \subset [n]$ if its
elements satisfy
$$
a_{i_1i_2\dots i_m}= 0, \forall i_1\in I, \forall i_2,\ldots, i_m\notin I.
$$
\end{definition}

The next result has been presented in \cite{LGX20} (see Corollary 2.8 there). 

\begin{thm}\label{LiRed}
Suppose that ${\mathcal A}$ is a nonsingular $M$-tensor and $b\in {\mathbb R}^n_+$ is nonzero. 
Then there is an index set $I\subseteq  I_0$ (which could be empty) such that every nonnegative solution to the following lower dimensional tensor
equation with $I_c = [n]\setminus I$ 
$$
{\mathcal A}_{I_c}x^{m-1}_{I_c}= b_{I_c}
$$
is positive. Moreover, every positive solution $x_{I_c}$ of the last equation together with $x_I = 0$
forms a nonnegative solution to equation \eqref{Teq}. 
\end{thm}

By comparing the discussions in \cite{LGX20} and in this paper, we can see that the index set $I\subseteq  I_0$ in Theorem \ref{LiRed} is the largest set such that 
 ${\mathcal A}$ is reducible with respect to $I$. 
We also see that the index set $I$ can be determined automatically by iteration \eqref{eqM} for computing the minimal nonnegative solution, with $x_0=0$.  
Indeed, $I=\{i \ | \ (x_k)_i=0\}$,  where  $k$ is the smallest integer such that the vectors $x_k$ and $x_{k+1}$ from iteration \eqref{eqM} have the same zero pattern (see Theorem \ref{rateMin}; $I=\emptyset$ when $b>0$). This is a very easy way to determine the set $I$. It would be much more expensive to determine $I$ by using the definition of reducibility to find 
the largest index set $I\subseteq  I_0$ such that ${\mathcal A}$ is reducible with respect to $I$. 

The numerical methods in \cite{LGX20} and \cite{LXG21} for computing a nonnegative solution of \eqref{Teq}  are based on Theorem \ref{LiRed}. 
They have quadratic convergence under suitable assumptions. For example, 
assuming equation \eqref{Teq} has already been reduced, the assumption needed for a Newton method in \cite{LXG21} is that 
for each $i\in I_0$, there is an element $a_{ii_2\ldots i_m}\ne 0$ with all $i_j\notin I_0$ ($j=2,\ldots, m$). So the assumption is stronger than our 
condition 3(d) 
in Theorem \ref{rateMin} for the linear convergence of our simple iteration for finding the minimal solution. 
From our comments on Theorem \ref{LiRed}, we know that the methods in \cite{LGX20} and \cite{LXG21}  can only find one of the nonnegative solutions that has the same number of zero elements as the minimal solution.
In particular, they will never find the maximal solution if it have more nonzero elements than the minimal solution. 
Moreover, to use the methods in \cite{LGX20} and \cite{LXG21}, the tensor ${\mathcal A}$ should be semi-symmetrized first, which increases computational work and makes the tensor much less sparse. 

\begin{example}
We consider Example \ref{ex1} with $k=1$. Then equation \eqref{Teq} has two solutions: $[0,1]^T$ and $[2,1]^T$. Note that $I_0=\{1\}$. 
We have $I=\{1\}$ for Theorem \ref{LiRed} since ${\mathcal A}$ is reducible with respect to $\{1\}$. The reduced equation is $x_2^3=1$ with nonnegative solution $x_2=1$, which is positive. A nonnegative solution of the original equation is then $[0,1]^T$ by Theorem \ref{LiRed}, but the other solution is lost in the reduction. 
We now apply Theorem \ref{rateMin}. With $x_0=[0,0]^T$, we get $x_1=[0,1]^T$ and $x_2=[0,1]^T$. So $k_0=1$ and $I=\{1\}$ in Theorem \ref{rateMin}. 
The reduced equation is $x_2^3=1$ with minimal nonnegative solution $x_2=1$, which is positive, and the minimal nonnegative solution of the original equation is then $[0,1]^T$ by Theorem \ref{rateMin} (we got this solution after just one iteration). 
We then apply Theorem \ref{rateMax}. With $x_0=[3,1]^T$, we get $x_1=[18^{1/3},1]^T$. So $k_0=0$ and $I=\{1,2\}$ in Theorem \ref{rateMax}. 
The equation is thus not reduced. We have $x_k=[s_k,1]^T$, where $s_k$ is determined by the iteration: $t_0=3$, $t_{k+1}=(2t_k^2)^{1/3}$. 
So $t_k$ converges to $2$ linearly with rate $2/3$, and $x_k$ converges to $[2,1]^T$ at the same rate. 
The linear convergence of $\{x_k\}$ is also guaranteed by Theorem \ref{rateMaxP} since condition 4 there is satisfied ($a_{1112}\ne 0$). 
\end{example}

There are examples for which iteration \eqref{eqM} converges linearly with a rate very close to $1$. We consider another extension of Example 1.1 in 
\cite{BHLZ21}. 

\begin{example}
We consider equation \eqref{Teq} with ${\mathcal A}\in {\mathbb R}^{[m,2]}$ given by $a_{11\ldots 1}=a_{22\ldots 2}=1$ and $a_{11\ldots 12}=-2$ and with 
$b=[0,1]^T$. Then equation \eqref{Teq} has two solutions: $[0,1]^T$ and $[2,1]^T$. Note that $I_0=\{1\}$. 
With $x_0=[3,1]^T$, we get $x_1=[(2\cdot 3^{m-2})^{1/(m-1)},1]^T$. So $k_0=0$ and $I=\{1,2\}$ in Theorem \ref{rateMax}. 
The equation is thus not reduced. We have $x_k=[s_k,1]^T$, where $s_k$ is determined by the iteration: $t_0=3$, $t_{k+1}=(2t_k^{m-2})^{1/(m-1)}$. 
So $t_k$ converges to $2$ linearly with rate $(m-2)/(m-1)$, and $x_k$ converges to $[2,1]^T$ at the same rate. 
The linear convergence of $\{x_k\}$ is also guaranteed by Theorem \ref{rateMaxP} since $a_{11\ldots 12}\ne 0$. 
However, when $m$ is large, the rate is very close to $1$, so $x_k$ converges to $[2,1]^T$ very slowly. 
\end{example}

\begin{remark}
We have the freedom to choose the splitting $M=P-Q$ for iteration \eqref{eqM}, when ${\mathcal A}$ is a nonsingular $M$-tensor. Recall that a smaller $Q$ is 
going to give faster termwise convergence (see Theorems \ref{ratecompm} and \ref{ratecompM}). For a dense tensor with order $m\ge 4$, computing ${\mathcal N}x^{m-1}$ will 
require $O(n^m)$ flops (for large $n$ and fixed $m$) and solving the linear system $My=c$ requires $O(n^3)$ flops. So it is advisable to use $P=M$ and $Q=0$ in the splitting $M=P-Q$ for a dense tensor with order $m\ge 4$. For a dense tensor with order $3$, it is advisable to take $P$ to be the lower triangular part of $M$ or the upper triangular part of $M$, since solving the linear system $Py=c$ requires $O(n^2)$ flops in this case. 
\end{remark}

\section{Conclusion}
\label{sec4}

We have presented new proofs for the existence of extremal nonnegative solutions of the $M$-tensor equation with a nonnegative right side vector by using some simple fixed-point iterations. We have studied these iterative methods and their generalizations.  With a suitable starting point, each of these methods has 
monotonic convergence (to the maximal nonnegative solution or to the minimal nonnegative solution) and the rate of convergence is (at least) linear under some mild assumptions. 
These methods are currently the only methods that are guaranteed to compute the maximal nonnegative solution or the minimal nonnegative solution with suitable initial guesses. There are examples for which the presented iterations have linear convergence at a rate very close to $1$. Iterative methods with faster convergence (without increasing computational work each iteration by too much) are still desirable. Those methods should be ones that can compute the maximal nonnegative solution and/or the minimal nonnegative solution, not just any one of the nonnegative solutions.

\section*{Acknowledgments} 
The author was supported in part by an NSERC Discovery Grant.

\end{document}